\documentclass[reqno]{siamltex704}
\usepackage{amsmath}
\usepackage{amssymb}
\usepackage{graphicx}
\usepackage{subfigure}
\newtheorem{remark}{Remark}[section]

\usepackage{color}

\newcommand{\vq}{{\bf q}}

\newcommand{\bx}{{\bf x}}
\newcommand{\vx}{{\bf x}}

\def\T{{\mathcal T}}
\def\E{{\mathcal E}}

\def\bn{{\bf n}}

\def\3bar{{|\hspace{-.02in}|\hspace{-.02in}|}}

\newcommand{\cV}{\mathcal{V}}
\newcommand{\cA}{\mathcal{A}}
\newcommand{\cB}{\mathcal{B}}

\newcommand{\ldp}{{(\hspace{-.03in}(}}
\newcommand{\rdp}{{)\hspace{-.03in})}}

\title{An auxiliary space multigrid preconditioner for the weak Galerkin method}
\author{Long Chen\thanks{Department of
Mathematics, University of California, Irvine -Irvine, CA 92697-3875 (chenlong@math.uci.edu). The research of Chen has been supported by NSF Grant DMS-1418934.}
\and Junping Wang\thanks{Division of Mathematical Sciences, National
Science Foundation, Arlington, VA 22230 (jwang@\break nsf.gov). The
research of Wang was supported by the NSF IR/D program, while
working at the Foundation. However, any opinion, finding, and
conclusions or recommendations expressed in this material are those
of the author and do not necessarily reflect the views of the
National Science Foundation.}
\and Yanqiu Wang\thanks{Department of Mathematics, Oklahoma State University,
Stillwater, OK 74075 (yqwang@math.okstate.edu).} \and Xiu
Ye\thanks{Department of Mathematics, University of Arkansas at
Little Rock, Little Rock, AR 72204 (xxye@ualr.edu).  The research of Ye was supported in part by National Science Foundation Grant DMS-1115097.}}

\begin{document}
\maketitle
\begin{abstract}
 In this paper, we construct an auxiliary space multigrid preconditioner for the weak Galerkin method
for second-order diffusion equations, discretized on simplicial 2D or 3D meshes.
The idea of the auxiliary space multigrid preconditioner is to use an auxiliary space as a ``coarse'' space in the multigrid algorithm,
where the discrete problem in the auxiliary space can be easily solved by an existing solver.
In our construction, we conveniently use the $H^1$ conforming piecewise linear finite element space as an auxiliary space.
The main technical difficulty is to build the connection between the weak Galerkin discrete space and the
$H^1$ conforming piecewise linear finite element space.
We successfully constructed such an auxiliary space multigrid preconditioner for the weak Galerkin method,
as well as a reduced system of the weak Galerkin method involving only the degrees of freedom on edges/faces.
The preconditioned systems are proved to have condition numbers independent of the mesh size.
Numerical experiments further support the theoretical results.
\end{abstract}



\section{Introduction}
Consider the second-order diffusion equation
\begin{equation} \label{eq:pde}
  \begin{aligned}
    -\nabla\cdot (\mathbb{A}\nabla u) &= f \qquad \textrm{in }\Omega,\\
    u &= 0        \qquad \textrm{on } \partial\Omega.
  \end{aligned}
\end{equation}
where $\Omega$ is a polygonal or polyhedral domain in
$\mathbb{R}^d\; (d=2,3)$. Assume that $\mathbb{A}$ is a symmetric,
uniformly positive definite, and uniformly bounded-above diffusion
matrix. Namely, there exist positive constants $\alpha$ and $\beta$
such that
\begin{equation}\label{matrix}
\alpha\xi^T\xi\leq \xi^T \mathbb A(x)\xi\le \, \beta\xi^T\xi \quad \text{for all }
\xi\in \mathbb{R}^d\;\textrm{and}\; x\in\Omega.
\end{equation}
The goal of this paper is to construct and analyze an auxiliary space multigrid preconditioner for the weak Galerkin
finite element discretization of Problem \eqref{eq:pde}.

The weak Galerkin method was recently introduced in
\cite{WangYe_PrepSINUM_2011} for second order elliptic equations. It
is an extension of the standard Galerkin finite element method where
classical derivatives were substituted by weakly defined derivatives
on functions with discontinuity. Optimal order of {\em a priori} error
estimates has been observed and established for various weak
Galerkin discretization schemes for second order elliptic equations
\cite{mwy-wg-stabilization, WangYe_PrepSINUM_2011, wy-mixed}.
An {\em a posteriori} error estimator was given in~\cite{Chen2013}.
Numerical implementations of weak Galerkin were discussed in
\cite{mwy-wg-stabilization, MuWangWangYe} for some model problems.

The weak Galerkin method has already demonstrated many nice
properties in various cases~\cite{WG-biharmonic, mwy-wg-stabilization, WangYe_PrepSINUM_2011, wy-mixed}.
Thus we are motivated to study fast solvers and preconditioning techniques for the weak Galerkin method.

The main results of this paper are:
\begin{itemize}
\item We develop a fast auxiliary space preconditioner for weak Galerkin methods using Raviart-Thomas element and Brezzi-Douglas-Marini element on triangular grids.
\item We consider both the original system and the reduced system. The original weak Galerkin discretization of \eqref{eq:pde} involves degrees of freedom
both on the interior of each mesh element and on mesh edges/faces.
The reduced system, which only involves degrees of freedom on edges/faces, is to our knowledge first rigorously constructed and analyzed here.
\end{itemize}

Recently, Li and Xie announced an auxiliary space mulrigrid preconditioning method for the weak Galerkin finite element method, and the result was posted on ArXive~\cite{LiXie}.  This result became to be known to us after the bulk portion of the present paper was developed. Following a through comparison, we conclude that our results are more general, and the two approaches are different in analysis. In addition, our result offers a new feature by covering the weak Galerkin method in the reduced system.

We shall briefly introduce the auxiliary space preconditioner constructed in~\cite{xu}.
A classical geometric multigrid method constructs discrete spaces on different mesh levels using the same type of
discretization. For example, in the classical multigrid method for $H^1$ conforming piecewise linear ($P_1$) finite element
approximation, one uses a set of nested meshes with characteristic mesh sizes $h$, $2h$, $4h$, $\ldots$, from the
finest mesh to the coarsest. An illustration of V-cycle multigrid is given in Figure \ref{fig:multigrid}.
The auxiliary space multigrid method can be essentially understood as a two-level method
involving a ``fine'' level and a ``coarse'' level, while the ``fine'' space and ``coarse'' space
are not necessarily using the same type of discretization or the same type of meshes.
This gives great freedom in choosing the ``coarse'' space, which is also called
an auxiliary space. Here we use the weak Galerkin discretization for the ``fine'' level,
and the $H^1$ conforming piecewise linear finite element discretization for the ``coarse'' level.
Both the ``fine'' level and the ``coarse'' level are discretized on the same mesh, as shown in Figure \ref{fig:multigrid}.
In the figure, we conveniently use black rectangles and black dots to denote different type of discretization spaces on different levels.
Because the fast solvers for the $H^1$ conforming piecewise linear finite element discretization
have been thoroughly studied, one can use any existing solvers/preconditioners as a ``coarse'' solver.
For example, one may use a classical multigrid method as a ``coarse'' solver and consequently achieves
a true ``multi''-grid effect (see Figure \ref{fig:multigrid}).

\begin{figure}
  \begin{center}
  \caption{Illustration of auxiliary space multigrid. We use black rectangles and black dots to denote different type of discretization spaces.
  The dashed circles shows how to derive an  auxiliary space ``multi''-grid method by using a classical multigrid as a coarse solver in the two-grid
  auxiliary space multigrid framework.} \label{fig:multigrid}
  \includegraphics[width=10cm]{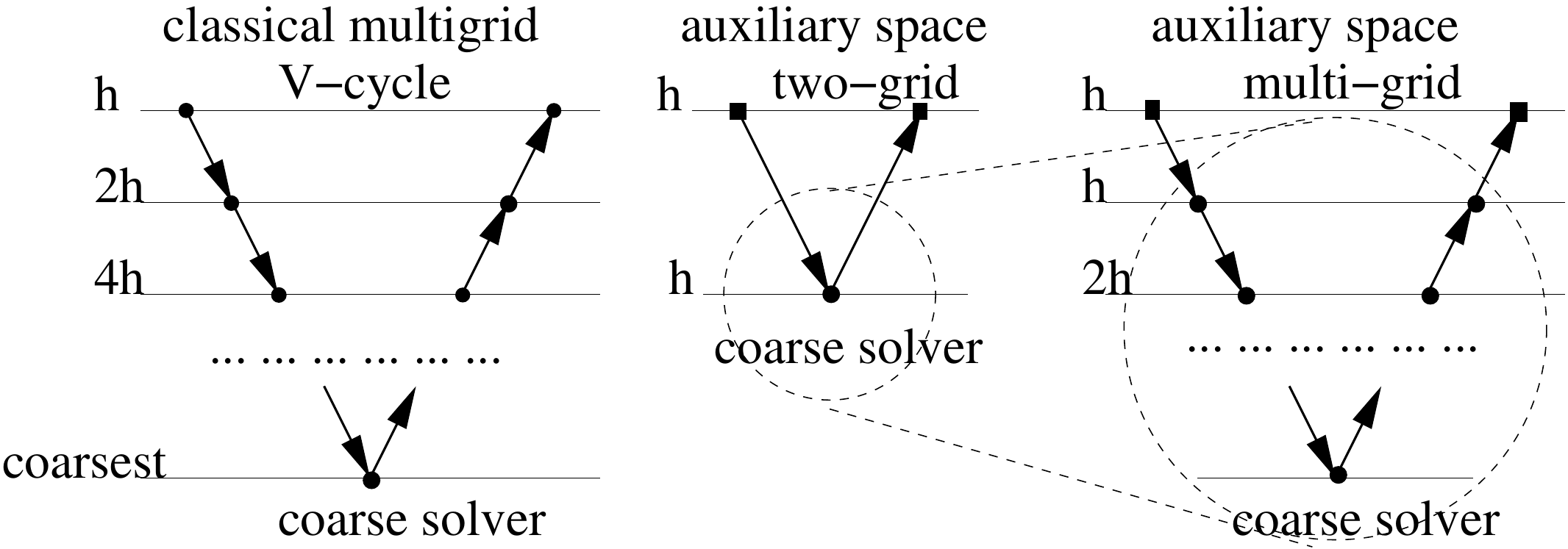}
  \end{center}
\end{figure}

The rest of the paper is organized as following. In Section
\ref{sec:weakGalerkin}, we give a brief introduction of the weak Galerkin
method, and in Section
\ref{sec:multigrid}, we construct the auxiliary space multigrid
preconditioner for the weak Galerkin discretization, and prove that
the condition number of the preconditioned system does not depend on
the mesh size. After that, we consider a reduced system of the weak Galerkin
discretization in Section \ref{sec:reduced} and construct an
auxiliary space multigrid solver/preconditioner
the reduced system, again with an optimal condition number estimate.
Finally in Section \ref{sec:numerical}, we present supporting numerical results.

\section{A Weak Galerkin Finite Element Scheme} \label{sec:weakGalerkin}
In this section, we give a brief introduction to the weak Galerkin method.
Related notation, definitions, and several important inequalities will be stated.

Let $D\subseteq\Omega$ be a polygon or polyhedron, we use the standard definition
of Sobolev spaces $H^s(D)$ and $H_0^s(D)$ with $s\ge 0$ (e.g., see
\cite{adams, ciarlet} for details). The associated inner product,
norm, and semi-norms in $H^s(D)$ are denoted by
$(\cdot,\cdot)_{s,D}$, $\|\cdot\|_{s,D}$, and $|\cdot|_{r,D}, 0\le r
\le s$, respectively. When $s=0$, $H^0(D)$ coincides with the space
of square integrable functions $L^2(D)$. In this case, the subscript
$s$ is suppressed from the notation of norm, semi-norm, and inner
products. Furthermore, the subscript $D$ is also suppressed when
$D=\Omega$. For $s<0$, the space $H^s(D)$ is defined to be the dual
of $H_0^{|s|}(D)$.

The above definition/notation can easily be extended to
vector-valued and matrix-valued functions. The norm, semi-norms, and
inner-product for such functions shall follow the same naming
convention. In addition, all these definitions can be transferred
from a polygonal/polyhedral domain $D$ to an edge/face $e$, a domain with lower
dimension. Similar notation system will be employed. For example,
$\|\cdot\|_{s,e}$ and $\|\cdot\|_e$ would denote the norm in
$H^s(e)$ and $L^2(e)$ etc. We also define the $H({\rm div})$ space as
follows
$$
H({\rm div},\Omega) = \{\vq:\ \vq \in [L^2(\Omega)]^d,\: \nabla\cdot\vq\in
L^2(\Omega)\}.
$$

Using the notation defined above, the variational form of Equation \eqref{eq:pde}
can be written as: Given $f\in L^2(\Omega)$, find $u\in H_0^1(\Omega)$ such that
\begin{equation} \label{eq:weakformulation}
(\mathbb A \nabla u, \, \nabla v) = (f, v)\qquad \textrm{for all } v\in H_0^1(\Omega).
\end{equation}

It is well known that equation \eqref{eq:weakformulation} admits a unique solution.
In addition, we assume that the solution to \eqref{eq:weakformulation} has $H^{1+s}$ regularity~\cite{Brenner, Grisvard85}, where $0<s\le 1$.
In other words, the solution $u$ is in $H^{1+s}(\Omega)$ and there exists a constant $C$ independent of $u$ such that
\begin{equation}
  \|u\|_{1+s} \le C\|f\|_0.
\end{equation}

Next, we present the weak Galerkin method for solving \eqref{eq:weakformulation}.
Let ${\cal T}_h$ be a
shape-regular, quasi-uniform triangular/tetrahedral mesh on the domain $\Omega$, with characteristic mesh size $h$.
For each triangle/tetrahedron $K\in {\cal T}_h$, denote by $K_0$ and $\partial K$ the interior and the
boundary of $K$, respectively. Geometrically, $K_0$ is identical to $K$.
Therefore, later in the paper, we often identify these two if it causes no ambiguity.
The boundary $\partial K$ consists of three edges in two-dimension, or four triangles in three-dimension.
Denote by $\E_h$ the collection of all edges/faces in ${\cal T}_h$. For
simplicity, throughout the paper, we use ``$\lesssim$''
to denote ``less than or equal to up to a general constant
independent of the mesh size or functions appearing in the
inequality''.

Let $j$ be a non-negative integer. On each $K\in {\cal T}_h$, denote
by $P_j(K_0)$ the set of polynomials with degree less than or equal
to $j$. Likewise, on each $e\in \E_h$, $P_j(e)$ is the set of
polynomials of degree no more than $j$. Following
\cite{WangYe_PrepSINUM_2011}, we define a weak discrete space on
mesh $\T_h$ by
$$
\begin{aligned}
V_{h} =  \{v:\:  v|_{K_0}\in P_j(K_0)\textrm{ for } K\in \T_h;
 \ v|_e\in P_l(e)\textrm{ for } e\in \E_h, &\\
 \textrm{and } v|_e=0 \textrm{ for } e\in \E_h\cap\partial\Omega\}, \text{ where }l = {j} \text{ or } j+1. &
\end{aligned}
$$
Observe that the definition of $V_h$ does not require any continuity
of $v\in V_h$ across interior edges/faces. A function in $V_h$ is
characterized by its value on the interior of each mesh element plus its
value on edges/faces. Therefore, it is convenient to represent
functions in $V_h$ with two components, $v=\{v_0, v_b\}$, where
$v_0$ denotes the value of $v$ on all $K_0$ and $v_b$ denotes the
value of $v$ on $\E_h$. The polynomial space $P_l(e)$ consists of two choices: $l= j$ or $j+1$ and the corresponding weak function space will sometimes be abbreviated as $W_{j,j}$ or $W_{j,j+1}$, respectively.

The weak Galerkin method seeks an approximation $u_h\in V_h$ to the
solution of problem \eqref{eq:weakformulation}. To this end, we
first introduce a discrete gradient operator, which is defined
element-wisely on each $K\in \T_h$. For the choices of $V_h$ given
above, i.e., using $W_{j,j}$ or $W_{j,j+1}$, suitable definitions of
the weak gradient involve the Raviart-Thomas (RT) element and the
Brezzi-Douglas-Marini (BDM) element, respectively. Let $K$ be either
a triangle or a tetrahedron and denote by $\widehat P_k(K)$ the set
of homogeneous polynomials of order $k$ in the variable
$\bx=(x_1,\ldots, x_d)^T$. Define the BDM element by
$G_j(K)=\left[P_{j+1}(K)\right]^d$ and the RT element by $G_{j}(K) =
\left[P_j(K)\right]^d + \widehat P_j(K) \bx$ for $j\geq 0$. Then, define a
discrete space
%
$$
\Sigma_h = \{\vq\in (L^2(\Omega))^d:\: \vq|_K \in G_j(K)\textrm{ for } K\in
\T_h\}.
$$
Here in the definition of $V_h$ and $\Sigma_h$, the RT element is
paired with $W_{j,j}$ while the BDM element is paired with
$W_{j,j+1}$. Note that $\Sigma_h$ is not necessarily a subspace of
$H({\rm div},\Omega)$, since it does not require any continuity in
the normal direction across mesh edges/faces.
\begin{definition}[Discrete Weak Gradient]
The discrete weak gradient of $v_h$ denoted by
$\nabla_{w}v_h$ is defined as the unique polynomial
$(\nabla_{w}v_h)|_K \in G_j(K)$ satisfying the following
equation
\begin{equation}\label{dwg}
(\nabla_{w}v_h, \vq)_K = -(v_0,\nabla\cdot \vq)_K+ \langle v_b,
\vq\cdot\bn\rangle_{\partial K}\quad \text{for all } \vq\in G_j(K),
\end{equation}
where $\bn$ is the unit outward normal on $\partial K$.
\end{definition}

Clearly, such a discrete weak gradient is always well-defined.
Furthermore, if $v\in H^1(K)$, i.e. $v_b = v_0|_{\partial K}$, and $\nabla v\in G_j(K)$, then one has $\nabla
_{w}v = \nabla v.$
%
In this paper we only consider the $W_{j,j}-RT$ and $W_{j,j+1}-BDM$
pairs on simplicial elements in the discretization. But there are many other different choices of discrete
spaces in the weak Galerkin method, defined on either simplicial
meshes or other types of meshes including general polytopal meshes~\cite{mwy-wg-stabilization}.
Extension of the multigrid preconditioner to other
weak Galerkin discretizations will be considered in the future work.

We define an $L^2$ projection from $H_0^1(\Omega)$ onto
$V_h$ by setting $Q_h v \equiv \{Q_0 v,\, Q_b v\}$, where $Q_0
v|_{K_0}$ is the local $L^2$ projection of $v$ to $P_j(K_0)$, for
$K\in\T_h$, and $Q_b v|_e$ is the local $L^2$ projection to
$P_l(e)$, for $e\in \E_h$. We also introduce $\mathbb Q_h$ the $L^2$ projection onto $\Sigma_h$. It is not hard to see the following operator identity~\cite{WangYe_PrepSINUM_2011}:
\begin{equation}\label{eq:QDDQ}
\mathbb Q_h \nabla u = \nabla _wQ_h u, \quad \text{ for all } u\in H_0^1(\Omega)
\end{equation}

For the $W_{j,j}-RT$ and $W_{j,j+1}-BDM$ pairs, it follows directly from \eqref{eq:QDDQ} that the discrete weak
gradient is a good approximation to the classical gradient, as
summarized in the following lemma~\cite{WangYe_PrepSINUM_2011}:

\medskip
\begin{lemma} \label{lem:assumptions}
  For any $v_h=\{v_0,\,v_b\}\in V_h$ and $K\in\T_h$, $\nabla_w v_h|_K = 0$ if and only if $v_0=v_b = constant$ on $K$.
Furthermore, for any $v\in H^{m+1}(\Omega)$, where $0\le m\le j+1$, we have
    $$\|\nabla_w (Q_h v)-\nabla v\|\lesssim h^{m}\|v\|_{m+1}.$$
In particular, for $v\in H^1(\Omega)$, the $L^2$-projection $Q_h$ is energy stable, i.e,
\begin{equation}\label{eq:H1stable}
\|\nabla _w (Q_h v)\|\lesssim \|\nabla v\| \quad \text{ for } v\in H^1(\Omega).
\end{equation}
\end{lemma}


Now we are able to present the weak Galerkin finite element formulation for \eqref{eq:weakformulation}:
Find $u_h = \{u_0,\, u_b\}\in V_h$ such that
  \begin{equation} \label{eq:wg}
a_h(u_h, v_h) = (f, v_0) \qquad \textrm{for all }v_h = \{v_0,\, v_b\}\in V_h,
  \end{equation}
where the bilinear form $a_h(\cdot,\cdot)$ on $V_h \times V_h$ is
defined by
\begin{equation}\label{eq:bilinearform}
a_h(u_h, v_h) :=   (\mathbb{A} \nabla_w u_h,\, \nabla_w v_h).
\end{equation}

The well-posedness and error estimates of the weak Galerkin formulation \eqref{eq:wg} have been discussed
in~\cite{WangYe_PrepSINUM_2011, WG-biharmonic}. To state these results, we first define a few discrete
inner-products and norms. For any $v_h=\{v_0, v_b\}$ and
$\phi_h=\{\phi_0, \phi_b\}$ in $V_h$, define a discrete $L^2$ inner-product by
$$
\ldp v_h, \phi_h\rdp \triangleq \sum_{K\in \T_h}\left [(v_0, \phi_0)_K + h ( v_0-v_b, \phi_0-\phi_b)_{\partial K}\right ].
$$
It is not hard to see that $\ldp v_h, v_h\rdp = 0$ implies
$v_h\equiv 0$. Hence, the inner-product is well-defined. Notice that
the inner-product $\ldp\cdot,\cdot\rdp$ is also well-defined for any
$v\in H^1(\Omega)$, for which $v_b|_e=v|_e$ is the trace
of $v$ on the edge $e$. In this case, the inner-product
$\ldp\cdot,\cdot\rdp$ is identical to the standard $L^2$
inner-product.

Define on each $K\in \T_h$
$$
\begin{aligned}
\| v_h \|_{0,h,K}^2 &= \|v_0\|_{0,K}^2 +  h \|v_0-v_b\|_{\partial K}^2, \\
\| v_h \|_{1,h,K}^2 &= \|v_0\|_{1,K}^2 +  h^{-1} \|v_0-v_b\|_{\partial K}^2, \\
| v_h |_{1,h,K}^2 &= |v_0|_{1,K}^2 +  h^{-1} \|v_0-v_b\|_{\partial K}^2.
\end{aligned}
$$
Using the above quantities, we define the following discrete norms
and semi-norms on the discrete space $V_h$
$$
\begin{aligned}
\| v_h \|_{0,h} &:= \left(\sum_{K\in \T_h}\| v_h \|_{0,h,K}^2 \right)^{1/2},\\
\| v_h \|_{1,h} &:= \left(\sum_{K\in \T_h}\| v_h \|_{1,h,K}^2 \right)^{1/2}, \\
| v_h |_{1,h} &:= \left(\sum_{K\in \T_h}| v_h |_{1,h,K}^2 \right)^{1/2}.
\end{aligned}
$$
It is clear that $\| v_h \|_{0,h}^2 = \ldp v_h, v_h\rdp$.
Moreover, we point out that the above norms and semi-norms are also well-defined for functions in $H^1(\Omega)$.
In this case they are identical to the usual $L^2$-norm, $H^1$-norm, and $H^1$-seminorm, respectively.

With the aid of the above defined norms, we state an additional
estimate of the $L^2$ projection $Q_h$, which was proved in
\cite{WG-biharmonic}.
\begin{lemma} \label{lem:Qh}
 For any $v\in H^{m}(\Omega)$ with $\frac{1}{2} < m \le j+1$, we have
  \begin{equation} \label{eq:a5}
    \|v-Q_h v\|_{0,h} \lesssim h^m \|v\|_{m}.
  \end{equation}
\end{lemma}


The following three Lemmas have also been proved in~\cite{WG-biharmonic}.
First, we have the equivalence between $\|\nabla_w(\cdot)\|$ and the $|\cdot|_{1,h}$ semi-norm:
\medskip
\begin{lemma} \label{lem:discretenorm-equivalence}
  For any $v_h=\{v_0, v_b\}\in V_h$, we have
  \begin{equation} \label{eq:discreteh1semi}
  |v_h|_{1,h} \lesssim \|\nabla_w v_h\| \lesssim | v_h|_{1,h}.
  \end{equation}
\end{lemma}

Moreover, the discrete semi-norms satisfy the usual inverse inequality, as
stated in the following Lemma.
\medskip
\begin{lemma} \label{lem:discreteinverse}
  For any $v_h=\{v_0, v_b\}\in V_h$, we have
  \begin{equation} \label{eq:discreteinverse}
    |v_h|_{1,h} \lesssim h^{-1} \|v_h\|_{0,h}.
  \end{equation}
Consequently, by combining \eqref{eq:discreteh1semi} and \eqref{eq:discreteinverse}, we have
\begin{equation}
  \|\nabla_w v_h\| \lesssim h^{-1} \| v_h\|_{0,h}.
\end{equation}
\end{lemma}

Next, the discrete semi-norm $\|\nabla_w(\cdot)\|$, which is equivalent to $|\cdot|_{1,h}$ as shown in
Lemma \ref{lem:discretenorm-equivalence}, satisfies a Poincar\'{e}-type inequality.
\medskip
\begin{lemma} \label{lem:discretepoincare}
The Poincar\'{e}-type inequality holds true for functions in
$V_{h}$. In other words, we have the following estimate:
\begin{equation}
\| v_h \|_{0,h} \lesssim \|\nabla_w v_h\|\qquad \textrm{for all } \ v_h\in V_{h}.\label{eq:discretepoincare1}
\end{equation}
\end{lemma}

Following the above lemmas and \eqref{matrix},
it is clear that equation \eqref{eq:wg} admits a unique solution.
This, together with error estimates for the weak Galerkin method, has been proved in~\cite{WangYe_PrepSINUM_2011}.

\medskip
\begin{theorem} \label{thm:apriori}
Assume Problem \eqref{eq:weakformulation} has $H^{1+s}$ regularity, where $0< s\le 1$.
The weak Galerkin problem \eqref{eq:wg} admits a unique solution.
Let $u\in H_0^1(\Omega)\cap H^{m+1}(\Omega)$, $0\le m\le j+1$, be the solution to \eqref{eq:weakformulation} and
$u_h = \{u_{h,0}, u_{h,b}\}$ be the solution to \eqref{eq:wg}, then we have
\begin{align}
  \|\nabla_w (Q_h u - u_h)\| &\lesssim h^m \|u\|_{m+1} , \\
  \|Q_0 u - u_{h,0}\| &\lesssim h^{m+s} \|u\|_{m+1} + h^{1+s} \|f-Q_0f\|.\label{eq:L2}
\end{align}
\end{theorem}

\begin{remark} \label{rem:apriori}
Theorem \ref{thm:apriori} is only stated for homogeneous Dirichlet boundary value problems.
Similar results hold for problems with non-homogeneous Dirichlet boundary or Neumann boundary conditions
\cite{WG-biharmonic, WangYe_PrepSINUM_2011}.
\end{remark}

At the end of this section, we state a scaled trace theorem. Let $K$ be an element with $e$ as an edge.  It is well known that
for any function $g\in H^1(K)$ one has
\begin{equation}\label{trace}
\|g\|_{e}^2 \lesssim h^{-1} \|g\|_K^2 + h \|\nabla
g\|_{K}^2.
\end{equation}

\section{An auxiliary space multigrid preconditioner} \label{sec:multigrid}
In this section, we construct an auxiliary space multigrid method for the weak Galerkin formulation \eqref{eq:wg}.
The auxiliary space multigrid method was introduced by J. Xu in~\cite{xu}.
Its main idea is to use an auxiliary space as a ``coarse'' space in the multigrid algorithm,
where the discrete problem in the auxiliary space can be easily solved by an existing solver.
In our construction, we will use the $H^1$ conforming piecewise linear finite element space as an auxiliary space.
The main technical difficulty is to build the connection between the weak Galerkin discrete space $V_h$ and the
$H^1$ conforming piecewise linear finite element space.

Define the auxiliary space $\cV_h \subset H_0^1(\Omega)$ to be $H^1$ the conforming piecewise linear finite element space on mesh
$\mathcal{T}_h$. The spaces $V_h$ and $\cV_h$ are equipped with inner-products $\ldp\cdot,\cdot\rdp$ and $(\cdot,\cdot)$, and induced
norms $\|\cdot\|_{0,h}$ and $\|\cdot\|$, respectively. Define linear
operators $A:\, V_h\rightarrow V_h$ and $\cA:\: \cV_h\rightarrow
\cV_h$ by
\begin{equation} \label{eq:Adefinition}
\begin{aligned}
\ldp A u,\, v \rdp &= (\mathbb{A} \nabla_w u, \, \nabla_w v)  \qquad &&\textrm{for all }v\in V_h, \\
(\cA u,\, v) &= (\mathbb{A} \nabla u,\, \nabla v) \qquad &&\textrm{for all }v\in \cV_h.
\end{aligned}
\end{equation}
By the Poincar\'{e} inequality and Lemma \ref{lem:discretepoincare},
it is clear that operators $A$ and $\cA$ are symmetric and positive definite with respect to $\ldp\cdot,\cdot\rdp$ and $(\cdot,\cdot)$, respectively.
Hence we can define the $A$-norm and $\cA$-norm on $V_h$ and $\cV_h$, respectively, by
$$
\begin{aligned}
  \|v\|_A &= \ldp A v,\, v\rdp^{1/2} = (\mathbb{A} \nabla_w v,\, \nabla_w v)^{1/2} \qquad &&\textrm{for all } v\in V_h, \\
  \|w\|_{\cA} &= ( \cA w,\, w)^{1/2} = (\mathbb{A} \nabla w,\, \nabla w)^{1/2} \qquad &&\textrm{for all } w\in \cV_h.
\end{aligned}
$$

It is well-known that the spectral radius and condition number of operator $\cA$ is $O(h^{-2})$~\cite{Bramble93}.
We have similar estimate for the operator $A$. Note that the authors of~\cite{LiXie} also give a proof
of the order of the condition number. But our proof is different from theirs and seems to be easier.

\medskip
\begin{lemma} \label{lem:spectralradius}
  The spectral radius of operator $A$, denoted by $\rho_A=\lambda_{\max}(A)$, and the condition number of operator $A$, denoted by $\kappa(A)$, are both of order $h^{-2}$.
\end{lemma}
\begin{proof}
By the definition of $A$ and Lemma \ref{lem:discreteinverse}, for all $v\in V_h$,
$$
\ldp A v,v\rdp \lesssim \|\nabla_w v\|^2 \lesssim h^{-2} \|v\|_{0,h}^2 = h^{-2} \ldp v,v\rdp.
$$
Because $A$ is symmetric and positive definite with respect to $\ldp \cdot,\cdot \rdp$, the above inequality implies that $\lambda_{\max}(A)\lesssim h^{-2}$. The discrete Poincar\'e inequality \eqref{eq:discretepoincare1} implies $\lambda_{\min}(A) \gtrsim 1$. Therefore $\kappa(A) = \lambda_{\max}(A)/\lambda_{\min}(A)\lesssim h^{-2}$.

To derive a lower bound for $\lambda_{\max}(A)$,
we first consider functions in $V_h$ with the form $v = \{0,v_b\}$. In other words, $v_0\equiv 0$.
Then, by the definition of discrete norms, Lemma \ref{lem:discretenorm-equivalence}
and the fact that $\mathbb{A}$ is uniformly positive definite, for such function $v$ we have
$$
\begin{aligned}
\ldp A v,v\rdp &\gtrsim \|\nabla_w v\|^2 \gtrsim |v|_{1,h}^2 = \sum_{K\in \mathcal{T}_h} h^{-1}\|v_b\|_{\partial K}^2 \\
&= h^{-2} \sum_{K\in \mathcal{T}_h} h\|v_b\|_{\partial K}^2= h^{-2} \|v\|_{0,h}^2 = h^{-2} \ldp v,v\rdp.
\end{aligned}
$$
Therefore, we must have $\lambda_{\max}(A)\gtrsim h^{-2}$. This implies the spectral radius $\rho_A = \lambda_{\max}(A) = O(h^{-2})$.

To get $\lambda_{\min}(A)\lesssim 1$, we chose the eigen-function $w$ of the smallest eigenvalue, $\lambda_1$,
 of $-\Delta$ with homogeneous Dirichlet boundary condition which satisfies $1=\|\nabla w\| = \sqrt{\lambda_1} \|w\|$.
It is well known that $\lambda_1 = O(1)$.
We then project $w$ to $V_h$ using the $L^2$-projection, i.e., $w_h  = Q_hw$. We estimate the norm of $w_h$ as follows: when $h$ is sufficiently small,
by the triangle inequality and Lemma \ref{lem:Qh} one has
$$
\|w_h\|\geq \|w\| - \|w-w_h\| \gtrsim \|w\| - Ch\|\nabla w\| = \|w\| - Ch \gtrsim \|w\|,
$$
where $C$ is a positive, general constant.
By the above inequality and the stability of $Q_h$ in the energy norm, c.f. \eqref{eq:H1stable}, we have
$$
\|w_h\|_A \lesssim \|\nabla w\| =\sqrt{\lambda_1} \|w\| \lesssim \|w_h\|.
$$
%
%
This completes the proof of the lemma.
\end{proof}

\smallskip
\begin{remark} \label{rem:linalgsystem}
By the triangle inequality, the trace inequality \eqref{trace} and the inverse inequality,
the norm $\|v_h\|_{0,h}$ is equivalent to $\left(\sum_{K\in \mathcal T_h}(\|v_0\|_K^2 + h\|v_b\|_{\partial K}^2)\right)^{1/2}$
in $V_h$. In practice, equation \eqref{eq:wg} can be written as a linear algebraic system by using
the canonical bases of $V_h$, i.e. Lagrange bases of $P_j(K)$ and $P_l(e)$ on each $K$ and $e$.
Using the standard scaling argument and the equivalent norm of $\|\cdot\|_{0,h}$,
it is not hard to see that for any $v_h \in V_h$, one has $\|v_h\|_{0,h}^2 \approx h^{d} \|\underline{v_h}\|_{l^2}^2$, where
$\underline{v_h}$ is the vector representation of $v_h$ under the canonical bases and $\|\cdot\|_{l^2}$ is the Euclidean norm
of vectors. Then, the stiffness matrix in the linear algebraic system resulting from \eqref{eq:wg}, i.e.,
the matrix representation of $\ldp A \cdot,\cdot\rdp$, also
has condition number of order $O(h^{-2})$. Thus it is not easy to solve equation \eqref{eq:wg} without efficient preconditioning.
\end{remark}
\smallskip

Next, we introduce the auxiliary space multigrid method for solving
equation \eqref{eq:wg}. The idea is to construct a multigrid method
using $V_h$ as the ``fine'' space and $\cV_h$ as the ``coarse''
space. Since $\cA$ is the discrete Laplacian on the conforming
piecewise linear finite element space, the ``coarse'' problem in
$\cV_h$ can be solved by many efficient, off-the-shelf solvers such
as the standard multigrid solver or a domain decomposition solver.
Denote $\cB:\: \cV_h\rightarrow \cV_h$ to be such a ``coarse''
solver. It can be either an exact solver or an approximate solver
that satisfies certain conditions, which will be given later. Next,
on the fine space, we need a ``smoother'' $R:\: V_h\rightarrow V_h$,
which is symmetric and positive definite. For example, $R$ can be a
Jacobi or symmetric Gauss-Seidel smoother. Finally, to connect the
``coarse'' space with the ``fine'' space, we need a ``prolongation''
operator $\Pi:\:\cV_h\rightarrow V_h$. A ``restriction'' operator
$\Pi^t:\: V_h \rightarrow \cV_h$ is consequently defined by
$$
( \Pi^t v, \,w) = \ldp v,\, \Pi w\rdp \quad\textrm{for } v\in V_h\textrm{ and } w\in \cV_h.
$$
Then, the auxiliary space multigrid preconditioner $B:\: V_h\rightarrow V_h$, following the definition in~\cite{xu, Bramble93},
is given by
\begin{align}
&\textrm{Additive} \qquad &&B = R + \Pi \cB \Pi^t,  \label{addB}\\
&\textrm{Multiplicative} \qquad && I-BA = (I-RA)(I-\Pi \cB \Pi^t)(I-RA). \label{mulB}
\end{align}

Both the additive and the multiplicative versions define symmetric multigrid solvers/preconditioners.
Readers may refer to~\cite{xu92} for the equivalence between symmetric solvers and preconditioners for symmetric problems.
Non-symmetric multiplicative
multigrid solver can similarly be defined but it cannot be used as a preconditioner. Thus we restrict our attention
to the symmetric version.

According to~\cite{xu}, the following theorem holds.
\medskip
\begin{theorem} \label{thm:abscond}
  Assume that for all $v\in V_h$, $w\in \cV_h$,
  \begin{align}
    \rho_A^{-1} \ldp v,\, v\rdp \lesssim \ldp R v, \, v\rdp &\lesssim \rho_A^{-1} \ldp v,\, v \rdp, \label{eq:mg-R}\\
    (\cA w,\, w) \lesssim (\cB \cA w,\, \cA w) &\lesssim (\cA w\, w), \label{eq:mg-cB}\\
    \|\Pi w\|_A &\lesssim \|w\|_{\cA} \quad\quad\textrm{(stability of $\Pi$)}, \label{eq:mg-Pi}
  \end{align}
and furthermore, assume that there exists a linear operator $P:\, V_h\rightarrow \cV_h$ such that
\begin{align}
  \|P v\|_{\cA} &\lesssim \|v\|_A, \quad\qquad\textrm{(stability of $P$)}\label{eq:mg-P1}\\
  \|v-\Pi P v\|_{0,h}^2 &\lesssim \rho_A^{-1} \|v\|_A^2 \qquad\textrm{(approximability)}.\label{eq:mg-P2}
\end{align}
Then the preconditioner $B$ defined in \eqref{addB} or \eqref{mulB} satisfies
$$
\kappa (BA) \lesssim O(1).
$$
\end{theorem}

\begin{remark}
Theorem \ref{thm:abscond} states that $B$ is a good preconditioner for $A$ as the condition
number of $BA$ is uniformly bounded. We thus can use the preconditioned conjugate gradient (PCG) method with $B$ being an effective preconditioner for solving the linear algebraic equation system associate to $Au = f$.
According to~\cite{xu92}, Theorem \ref{thm:abscond} also implies that $I-\omega BA$,
where $0<\omega<2/\rho_{BA}$, defines an efficient iterative solver.
\end{remark}
\medskip

Now we shall construct an auxiliary space preconditioner which satisfies all conditions in Theorem \ref{thm:abscond},
namely, inequalities \eqref{eq:mg-R}-\eqref{eq:mg-P2}.
It is straight forward to pick $\cB$ that satisfies condition \eqref{eq:mg-cB}.
For example, $\cB$ can be either the direct solver, for which $\cB\sim \cA^{-1}$,
or one step of classical multigrid iteration~\cite{Bramble93} which satisfies condition \eqref{eq:mg-cB}.

The smoother $R$ is also easy to define. In view of Remark \ref{rem:linalgsystem}, a Jacobi or a symmetric Gauss-Seidel smoother~\cite{Bramble93} will satisfy condition \eqref{eq:mg-R}.
Hence it remains to construct operators $\Pi$ and $P$ that satisfy the conditions \eqref{eq:mg-Pi}-\eqref{eq:mg-P2}.

The operator $\Pi$ is actually easy to choose, and we simply define
$\Pi=Q_h = \{Q_0,\, Q_b\}$. Note when $V_h$ consists of $W_{j,j}$
elements or $W_{j,j+1}$ elements with $j\ge 1$, it is clear that for
all $w\in \cV_h$ and $K\in \mathcal{T}_h$, $(\Pi w)|_K =
\{w|_{K_0},\, w|_{\partial K}\}$ which is just the natural inclusion
of $\mathcal V_h$ into $V_h$.
The stability of $\Pi$ in the energy norm follows immediately from
\eqref{eq:H1stable} and the boundedness of the diffusion coefficient $\mathbb A$:
\smallskip
\begin{lemma}
Let $\Pi=Q_h = \{Q_0,\, Q_b\}$. Then $\Pi$ satisfies condition \eqref{eq:mg-Pi}, i.e.,
$$
\|\Pi w\|_A \lesssim \|w\|_{\cA}, \quad \text{ for all } w\in \cV_h.
$$
\end{lemma}

Next, we construct an operator $P$ that satisfies \eqref{eq:mg-P1} and \eqref{eq:mg-P2}.
\begin{definition}
  Let $0\le \alpha_1,\alpha_2,\ldots, \alpha_k\le 1$ satisfy $\sum_{i=1}^k \alpha_i = 1$, and
let $\{c_1,c_2,\ldots, c_k\}$ be a sequence of numbers.
The value $\sum_{i=1}^k \alpha_i c_i$ is called a convex combination of $\{c_1,c_2,\ldots, c_k\}$.
\end{definition}
\medskip

A function in $\cV_h$ is completely determined by its value on mesh vertices.
Let $v=\{v_0,v_b\} \in V_h$. To define $Pv$, one only needs to specify its value on all mesh vertices.
Hence we can define $P$ as follows: on each mesh vertex $\vx$,
the value of $Pv (\vx)$ is a prescribed convex combination of the values of $v_0(\vx)$ and $v_b(\vx)$ on all mesh elements and edges/faces
that have $\vx$ as a vertex.
Moreover, to preserve the homogeneous boundary condition, when $\vx\in\partial\Omega$,
the convex combination shall be constructed such that
it only depends on the value of $v_b(\vx)$ on boundary edges/faces that have $\vx$ as a vertex.
Of course, for problems with the homogeneous Dirichlet boundary condition, one can simply set $Pv(\vx) = 0$ on boundary vertices.
But the current set-up would allow easy extension to non-homogeneous boundary conditions.

\medskip
\begin{lemma}
  Operator $P$ satisfies
\begin{equation} \label{eq:Pvapproximability}
  \|v-Pv\|_{0,h}^2 + h^2 |v-Pv|_{1,h}^2 \lesssim h^2 |v|_{1,h}^2,\qquad \textrm{for all } v\in V_h.
\end{equation}
\end{lemma}
\begin{proof}
For each $K\in \mathcal{T}_h$, denote by $V(K)$ the vertices of $K$.
For each $K\in \mathcal{T}_h$ and $v=\{v_0,v_b\} \in V_h$,
denote by $I_{h,K}v_0$ the nodal value interpolation of $v_0$ into $P_1(K)$,
i.e., $I_{h,K}v_0\in P_1(K)$ and is identical to $v_0$ on $V(K)$.
By the approximation property of nodal value interpolations,
the scaling argument, the definition of $P$, the triangle inequality,
and the finite overlapping property of quasi-uniform meshes, we have
$$
\begin{aligned}
&\|v-P v\|_{0,h}^2 \\
=& \sum_{K\in \mathcal{T}_h}  \left(|v_0-P v|_{0,K}^2 + h\|v_0-v_b\|_{0,\partial K}^2 \right) \\
\lesssim & \sum_{K\in \mathcal{T}_h} \left(|v_0-I_{h,K} v_0|_{0,K}^2  + \sum_{\vx\in V(K)} h^2 |I_{h,K} v_0 (\vx)-Pv(\vx)|^2 + h\|v_0-v_b\|_{0,\partial K}^2 \right)\\
\lesssim & \sum_{K\in \mathcal{T}_h} \left(h^2|v_0|_{1,K}^2 + \sum_{\vx\in V(K)} h^2 |v_0 (\vx)-v_b(\vx)|^2 + h\|v_0-v_b\|_{0,\partial K}^2 \right)\\
\lesssim & \sum_{K\in \mathcal{T}_h}  \left(h^2|v_0|_{1,K}^2 + h \|v_0 - v_b\|_{0,\partial K}^2\right) \\
=& h^2 |v|_{1, h}^2.
\end{aligned}
$$
Combining the above with the inverse inequality \eqref{eq:discreteinverse} completes the proof of the lemma.
\end{proof}
\medskip

\begin{lemma} \label{lem:Pstability}
The operator $P$ satisfies the properties \eqref{eq:mg-P1} and \eqref{eq:mg-P2}.
\end{lemma}
\begin{proof}
By using inequalities \eqref{eq:Pvapproximability} and \eqref{eq:discreteh1semi}, for all $v\in V_h$, we have
$$
\|P v\|_{\cA}^2 \lesssim |Pv |_{1,h}^2
\lesssim |v-Pv|_{1,h}^2 + |v|_{1,h}^2
\lesssim |v|_{1,h}^2
\lesssim \|v\|_{A}^2.
$$
This completes the proof of Inequality \eqref{eq:mg-P1}.

We then estimate $\|Pv-\Pi P v\|_{0,h}$.
When $j\geq 1$, $\|Pv-\Pi P v\|_{0,h} = 0$ since $\Pi$ is the natural inclusion.
We only need to consider the case $j = 0$.
Since $\Pi Pv$ is the average of $Pv$, we get
$$
\|Pv - \Pi Pv\|_{K}\lesssim h|Pv|_1, \quad \|Pv - \Pi Pv\|_{\partial K}\lesssim h|Pv|_{1,\partial K},
$$
by the average type Poincar\'e inequality. By the scaled trace inequality \eqref{trace} and the fact $|Pv|_{2,K} = 0$ for a piecewise linear function, we can bound $h^{1/2}|Pv|_{1,\partial K} \lesssim |Pv|_{1,K}$. Therefore, we obtain
$$
\|Pv-\Pi P v\|_{0,h} \lesssim h|Pv|_{1}= h |Pv|_{1,h}\lesssim h|v|_{1,h}.
$$
Then,  by the triangle inequality and the coercivity of operator $A$, for all $v\in V_h$, we have
$$
\begin{aligned}
\|v-\Pi P v\|_{0,h} &\lesssim \|v-Pv\|_{0,h} + \|Pv-\Pi P v\|_{0,h}\lesssim h|v|_{1,h} \lesssim h\|v\|_A.
\end{aligned}
$$
Combining the above with the estimate $\rho_A = O(h^{-2})$ (see Lemma \ref{lem:spectralradius}), this completes the proof of Inequality \eqref{eq:mg-P2}.
\end{proof}

\begin{remark}
In the proof of Lemma \ref{lem:Pstability}, one may also use Lemma \ref{lem:Qh} and the Poincar\'e inequality to
estimate $\|Pv-\Pi P v\|_{0,h}$, i.e.,
$$
\|Pv-\Pi P v\|_{0,h} \lesssim h\|Pv\|_1 \lesssim h|Pv|_{1}.
$$
This requires the Poincar\'e inequality for $Pv$, which is not true for non-homogeneous Dirichlet boundary problems.
The current approach avoids such difficulty and can thus be easily
extended to non-homogeneous Dirichlet boundary problems or Neumann boundary problems.
\end{remark}
\smallskip

By now, all conditions in Theorem \ref{thm:abscond} have been verified for the given multigrid construction. We summarize it in the following theorem:
\begin{theorem}
Suppose we have a smoother $R$ and an auxiliary solver $\mathcal B$ satisfying the property: for all $v\in V_h$, $w\in \cV_h$,
\begin{align*}
\rho_A^{-1} \ldp v,\, v\rdp \lesssim \ldp R v, \, v\rdp &\lesssim \rho_A^{-1} \ldp v,\, v \rdp,\\
    (\cA w,\, w) \lesssim (\cB \cA w,\, \cA w) &\lesssim (\cA w\, w).
\end{align*}
Let $B = R + \Pi \cB \Pi^t$ or defined implicitly by the relation $I-BA = (I-RA)(I-\Pi \cB \Pi^t)(I-RA)$. Then $B$ is symmetric and positive definite and
$\kappa (BA) \lesssim O(1)$.
\end{theorem}

\medskip
\begin{remark}
The operator $P$, although its definition seems to be complex, is only needed in the theoretical analysis.
In the implementation, one only needs $\cB$, $R$ and $\Pi$. It is also well-known that
the matrix representation of the restriction operator $\Pi^t$ is just the transpose of the
matrix representation of the prolongation operator $\Pi$.
\end{remark}

\section{Reduced system and its multigrid preconditioner} \label{sec:reduced}
By using the Schur complement, the weak Galerkin problem \eqref{eq:wg} can be reduced to a system involving only the degrees of freedom
on mesh edges/faces. In this section, we present such a reduced system and construct an auxiliary space multigrid
preconditioner for the reduced system.

\subsection{Reduced system}
Let
$$
\begin{aligned}
  V_0 &= \{v\:|\: v = \{v_0,\,0\}\in V_h\}, \\
  V_b &= \{v\:|\: v = \{0,\, v_b\}\in V_h\},
\end{aligned}
$$
be two subspaces of $V_h$. Clearly one has $V_h = V_0+V_b$.
For any function  $v = \{v_0,\,v_b\}\in V_h$, it is convenient to extend the notation of $v_0$ and $v_b$ so that, without ambiguity, $v_0\in V_0$ and $v_b\in V_b$.
Functions in $V_0$ and $V_b$ will also often be referred to as $v_0$ and $v_b$, respectively.

Then Equation \eqref{eq:wg} can be rewritten into
\begin{equation} \label{eq:reduced1}
\begin{cases}
  a_h(u_0, \, v_b) + a_h(u_b, \, v_b) = 0, \quad \qquad \text{ for all } v_b\in V_b,\\
 a_h(u_0, \, v_0) + a_h(u_b, \, v_0) = (f,\, v_0)\quad \text{ for all } v_0\in V_0.
\end{cases}
\end{equation}

By choosing a basis of $V_h$, we can obtain a matrix form of \eqref{eq:reduced1}. Let $\mathbf v$ be the vector representation of a weak function $v\in V_h$ and $\mathbf M$ be the matrix representation of an operator $M$ relative to the chosen basis. We can write the matrix form of \eqref{eq:reduced1} as follows
\begin{equation}\label{eq:originalmatrix}
\begin{pmatrix}
\mathbf A_b & \mathbf A_{b0}\\
\mathbf A_{0b} & \mathbf A_{0}
\end{pmatrix}
\begin{pmatrix}
\mathbf u_b\\
\mathbf u_0
\end{pmatrix}
=
\begin{pmatrix}
0\\
\mathbf f
\end{pmatrix}.
\end{equation}
Note that $\mathbf A_0$ is block-diagonal. We can thus solve $\mathbf u_0$ from the second equation and substitute into the first equation to obtain the Schur complement equation
\begin{equation}\label{eq:Schur}
(\mathbf A_b - \mathbf A_{b0}\mathbf A_0^{-1}\mathbf A_{0b}) \mathbf u_b = - \mathbf A_0^{-1}\mathbf f.
\end{equation}
After $\mathbf u_b$ is obtained by solving \eqref{eq:Schur}, the interior part $\mathbf u_0 = \mathbf A_0^{-1}(\mathbf f - \mathbf A_{0b}\mathbf u_b)$ can be computed element-wise.

The reduced system \eqref{eq:Schur} involves less degrees of freedom than the original weak Galerkin system \eqref{eq:originalmatrix}.
Indeed, the difference between these two degrees of freedom is exactly $\dim(V_0)$, which is equal to
$(j+1)(j+2)/2$ times the total number of mesh triangles in two-dimension,
and $(j+1)(j+2)(j+3)/6$ times the total number of mesh tetrahedron in three-dimension. More importantly, the Schur complement $\mathbf A_b - \mathbf A_{b0}\mathbf A_0^{-1}\mathbf A_{0b}$ is also a SPD matrix and has the same sparsity as $\mathbf A_b$. Therefore solving the reduced system \eqref{eq:Schur} is more efficient than solving the original system \eqref{eq:originalmatrix} provided a good preconditioner for \eqref{eq:Schur} is available. In the rest of this section, we will construct a fast auxiliary multigrid preconditioner for \eqref{eq:Schur}. Note that the algorithm is implemented in the matrix formulation. The analysis, however, is given in the operator form. In the following we will introduce corresponding operators.


We first introduce an $a_h(\cdot, \cdot)$-orthogonal projector $P_0$ from $V_b$ to $V_0$ as follows: For $v_b\in V_b$, define $P_0v_b\in V_0$ such that
$$
a_h(P_0 v_b,\, \zeta_0 ) = a_h(v_b,\, \zeta_0 ) \qquad\textrm{for all }  \zeta_0\in V_0.
$$
It is not hard to see that $\|(I-P_0)v_b\|_{0,h} = \|\{-P_0v_b,v_b\}\|_{0,h}$ is a well-defined norm on $V_b$.
In the following analysis we shall always equip $V_b$ with this new norm and $V_0$ with the inherited norm $\|\cdot\|_{0,h}$.
By the trace inequality, the inverse inequality and the definition of $\|\cdot\|_{0,h}$, one has
$$
\|P_0v_b\|_{0,h} \lesssim \|P_0 v_b\| \lesssim \|\{-P_0v_b,v_b\}\|_{0,h} = \|(I-P_0)v_b\|_{0,h},
$$
which implies that $P_0:\: V_b\rightarrow V_0$ is a bounded linear operator under the newly assigned norms.
Denote by $V_0'$ and $V_b'$ the space of bounded linear functionals on $V_0$ and $V_b$, respectively.
Then the bounded linear operator $P_0$ induces a bounded dual operator $P_0': V_0'\to
V_b'$, i.e.,  for $F\in V_0'$, $\langle P_0'F, v_b \rangle  \triangleq
\langle F, P_0 v_b\rangle $ for all $v_b \in V_b$. In particular,
let $F$ be defined by $\langle F, \cdot\rangle = (f, \cdot)$ for $f\in L^2(\Omega)$,
then one has $\langle P_0'F, v_b \rangle = (f, P_0v_b)$.

We claim, and will prove later, that the operator form of the Schur complement equation \eqref{eq:Schur} is
\begin{equation} \label{eq:reduced}
\begin{aligned}
  a_h((I - P_0) u_b,\, v_b)  = -\langle P_0'F,  \, v_b\rangle, \quad \text{for all } v_b\in V_b.
\end{aligned}
\end{equation}
Note that by the property of the projection $P_0$, Equation \eqref{eq:reduced} can also be written
into the symmetric form $a_h((I - P_0) u_b,\, (I - P_0)v_b)  = -\langle P_0'F,  v_b\rangle$ for all $v_b\in V_b$.

To prove this, we first
define a linear operator $A_0^{-1}: L^2(\Omega)\rightarrow V_0$ by: for a function $g \in L^2(\Omega)$, one has $A_0^{-1} g\in V_0$ such that
$$
a_h(A_0^{-1} g, \, v_0) = (g,\, v_0)\qquad \textrm{for all }v_0\in V_0.
$$
The well-posedness of $A_0^{-1}$ follows directly from the
coercivity of $a_h(\cdot,\cdot)$ on $V_h$, and consequently on its
subspace $V_0$. Moreover, the restriction of $A_0^{-1}$ to $V_0$ is
symmetric and positive definite.
Noticing that $\nabla_w v_0$ is locally defined on each mesh element,
it is clear that $A_0^{-1}$ is also locally defined on each mesh element.

Denote by $\nabla_h\cdot$ the piecewise divergence operator on
$\Sigma_h$, and by $\mathbb Q_h:\: L^2(\Omega)^d \rightarrow \Sigma_h$ the
$L^2$ projection. Using the above notation and the definition of
$\nabla_w$ , the second equation in \eqref{eq:reduced1} implies that
\begin{equation} \label{eq:u0}
\begin{aligned}
  (\mathbb{A} \nabla_w u_0,\,\nabla_w v_0) &=
  (f,\, v_0) - (\mathbb{A} \nabla_w u_b, \, \nabla_w v_0) \\
  &= (f,\, v_0) + (\nabla_h\cdot (\mathbb Q_h\mathbb{A} \nabla_w u_b),\, v_0),
\end{aligned}
\end{equation}
which leads to
\begin{equation} \label{eq:u0-2}
u_0 = A_0^{-1} (f + \nabla_h\cdot (\mathbb Q_h\mathbb{A} \nabla_w u_b) ).
\end{equation}

\medskip
Next, we note that the projection $P_0$ is identical to $-A_0^{-1} \nabla_h\cdot (\mathbb Q_h\mathbb{A} \nabla_w)$ on $V_b$:
\begin{lemma} \label{lem:orthogonal}
The orthogonal operator $P_0:\:V_b\rightarrow V_0$
$$
P_0 v_b = -A_0^{-1} \nabla_h\cdot (\mathbb Q_h\mathbb{A} \nabla_w v_b)\qquad \textrm{for all } v_b\in V_b.
$$
\end{lemma}
\begin{proof}
By the definition of weak gradient $\nabla_w$ and $A_0^{-1}$, we have
\begin{align*}
(\mathbb{A} \nabla_w v_b,\, \nabla_w \zeta_0 ) &= - (\nabla_h\cdot (\mathbb Q_h\mathbb{A} \nabla_w v_b),\, \zeta_0 ) \\
&=  -(\mathbb{A}\nabla_w A_0^{-1} \nabla_h\cdot (\mathbb Q_h\mathbb{A} \nabla_w v_b),\, \nabla_w \zeta_0 ).
\end{align*}
  By the definition of $P_0$, we then complete the proof of the lemma.
\end{proof}

\smallskip
\begin{remark}
The operator $P_0$ corresponds to the matrix $\mathbf A_0^{-1}\mathbf A_{0b}$.
\end{remark}
\smallskip

Now, by \eqref{eq:u0-2} and Lemma \ref{lem:orthogonal}, one has $u_0=A_0^{-1} f - P_0 u_b$.
Substituting this into the first equation of \eqref{eq:reduced1} gives
\begin{align*}
  a_h((I-P_0)u_b,\, v_b) &= -a_h(A_0^{-1}f, \,v_b)
  = -a_h(A_0^{-1}f, \,P_0 v_b) \\
&  = -(f, \, P_0 v_b) = -\langle P_0'F, v_b \rangle.
\end{align*}

This completes the derivation of the reduced problem \eqref{eq:reduced} from the original problem \eqref{eq:wg}.
Here we emphasize again that $P_0'F \in V_b'$ is bounded in the sense that
\begin{equation} \label{eq:boundedlf}
|\langle P_0'F, v_b \rangle| \lesssim  \|(I-P_0)v_b\|_{0,h}.
\end{equation}

We will further reformulate the reduced system \eqref{eq:reduced}.
To this end, we define a subspace of $V_h$ as $V_r = \{ v_r\,  | \,
v_r = (I-P_0)v_b = \{-P_0 v_b, v_b\}\textrm{ for all } v_b \in V_b\}$, which is just the
graph of $V_b$ under $I - P_0$.
The space $V_r$ inherits the norm $\|\cdot\|_{0,h}$ from $V_h$,
and hence $V_r$ and $V_b$ (equipped with the norm $\|(I-P_0)\cdot\|_{0,h}$)
are clearly isomorphic under the mapping $I-P_0:\: V_b\rightarrow V_r$.
Moreover, the right-hand side of Equation \eqref{eq:reduced} can be written into
$$
\begin{aligned}
-\langle P_0'F,  \, v_b\rangle &= -\langle \{0,P_0'F\},\, \{-P_0 v_b,v_b\}\rangle
= -\langle \{0,P_0'F\},\, v_r \rangle \\
&\triangleq \langle \mathcal{F}, \, v_r \rangle,
\end{aligned}
$$
where $\mathcal{F}$ is a bounded linear functional on $V_r$ according to \eqref{eq:boundedlf}.

By using Lemma \ref{lem:orthogonal} and combining the above analysis, Equation \eqref{eq:reduced} can now
be rewritten into: Find $u_r\in V_r$ such that
\begin{equation}\label{eq:reducedeq}
a_h(u_r,\, v_r) = \langle \mathcal{F}, v_r\rangle, \quad \text{for all
}v_r\in V_r.
\end{equation}
The well-posedness of \eqref{eq:reducedeq} then follows from
the continuity and coercivity of the bilinear form $a_h(\cdot,
\cdot)$ restricted to $V_r$ and the fact that $\mathcal{F}$ is a
bounded linear functional on $V_r$.

\subsection{Auxiliary space preconditioner for the reduced system}
Now we are able to consider an auxiliary space multigrid preconditioner for the reduced system \eqref{eq:reducedeq},
using again the $H^1$ conforming piecewise linear finite element space as the auxiliary space.
Denote by $A_r$ the restriction of operator $A$, defined in \eqref{eq:Adefinition}, to the subspace $V_r$.
That is, $A_r:\: V_r\rightarrow V_r$ is defined by
$$
\ldp A_r u,\, v\rdp = a_h(u,\, v)\quad\textrm{for all }v\in V_r.
$$

To apply Theorem \ref{thm:abscond}, we define a prolongation operator $\Pi_r:\: \cV_h\rightarrow V_r$ and a linear operator $P_r:\: V_r\rightarrow\cV_h$ by
$$
\Pi_r = (I-P_0)Q_b \quad\textrm{and}\quad P_r = P|_{V_r}.
$$
\begin{lemma}
Both $\Pi_r$ and $P_r$ are stable in the energy norm, i.e.,
\begin{align}
\label{Pistable}\| \Pi_r v \|_{A} \lesssim \|v\|_{\cA},\qquad\textrm{for all }v\in\cV_h\\
\label{Pstable} \|P_r v_r\|_{\cA} \lesssim \|v_r\|_A,\qquad\textrm{for all }v_r\in V_r.
 \end{align}
\end{lemma}

\begin{proof}
The stability of $\Pi_r$ follows from the property of $P_0$ and the stability \eqref{eq:H1stable} of $Q_h$:
$$
\begin{aligned}
\| \Pi_r v \|_{A}^2 &= \|(I-P_0)Q_b v\|_A^2 = (\mathbb{A} \nabla_w (I-P_0)Q_b v,\, \nabla_w (Q_bv +Q_0 v)) \\
&\lesssim \| \Pi_r v \|_{A} \|Q_h v\|_A \lesssim  \| \Pi_r v \|_{A} \|v\|_{\cA}.
\end{aligned}
$$
The stability of $P_r$ simply follows
from that of $P$.
This completes the proof of the lemma.
\end{proof}

To verify the approximation property, we first explore the relation between $Q_hw$ and $\Pi_r w$ for $w\in \mathcal V_h$.
It turns out that $Q_hw = \Pi_r w$ for all $w\in \mathcal V_h$ when the diffusion coefficient matrix $\mathbb A$ is piecewise constant.

\begin{lemma}
When $\mathbb A$ is piecewise constant, we have for all $w\in \mathcal V_h$,
$$
Q_hw = \Pi_r w.
$$
\end{lemma}
\begin{proof}
Recall that $\Pi_r w = (I-P_0) Q_b w$. Since $P_0$ is the orthogonal projection, we have
$$
(\mathbb{A}\nabla_w\Pi_r w ,\, \nabla_w \zeta_0) = (\mathbb{A}\nabla_w (I-P_0) Q_b w),\, \nabla_w \zeta_0) = 0\qquad\textrm{for all }\zeta_0\in V_0.
$$
On the other hand, using the relation \eqref{eq:QDDQ} and the fact that both $\nabla w$ and $\mathbb A$ are piecewise constant,
\begin{align*}
(\mathbb{A} \nabla_w Q_h w, \, \nabla_w \zeta_0)_K  = (\mathbb{A}\mathbb Q_h \nabla  w, \, \nabla_w \zeta_0)_K = (\mathbb{A}\nabla  w, \, \nabla_w \zeta_0)_K\\
= - (\nabla_h\cdot(\mathbb{A} \nabla w),\, \zeta_0)_K = 0, \qquad\textrm{for all }\zeta_0\in V_0.
\end{align*}
Therefore
\begin{equation}\label{orth}
a_h(\Pi_r w - Q_h w, \zeta_0) = 0, \quad \text{for all } \zeta_0\in V_0
\end{equation}
The fact $\Pi_r w - Q_h w \in V_0$ and the orthogonality~\eqref{orth} implies $\Pi_r w= Q_h w$.
\end{proof}

\medskip

Similar to the analysis in Section \ref{sec:multigrid}, we can establish the following results.

\begin{lemma} \label{lem:reducedspectralradius}
Suppose $\mathbb{A}$ is piecewise constant and the space $V_r$ is non-trivial, i.e., the triangulation contains at least one interior vertex.
Then the spectral radius of operator $A_r$, denoted by $\rho_{A_r}$, is of order $h^{-2}$.
\end{lemma}
\begin{proof}
Recall that $$\rho_{A_r} = \lambda_{\max}(A_r) = \max_{v\in V_r}\frac{\ldp A_r v, v\rdp}{\ldp v, v\rdp} = \max_{v\in V_r}\frac{\ldp Av, v\rdp }{\ldp v, v\rdp}.$$
Since $V_r\subset V_h$, we immediately get $\rho_{A_r}\leq \rho_A\lesssim h^{-2}$.

To show the lower bound, we pick a hat function $w\in \mathcal V_h$. By the standard scaling argument,
\begin{equation}\label{eq:wscaling}
\|w\|\lesssim h|\nabla w|.
\end{equation}
We then chose $v = \Pi_r w \in V_r$ and estimate its norms. First
\begin{equation}\label{eq:vl2}
\|v\|_{0,h} = \|\Pi_r w\|_{0,h} = \|Q_h w\|_{0,h}\lesssim \|w\|.
\end{equation}
Second, as $\nabla w$ is piecewise constant, $\mathbb Q_h \nabla w = \nabla w$ and
\begin{equation}\label{eq:vh1}
\|\nabla w\| = \|\mathbb Q_h \nabla w\| = \|\nabla _w Q_h w\|  = \|\nabla _w \Pi_r w\| \lesssim (A_r v, v) ^{1/2}.
\end{equation}
Combining \eqref{eq:vl2}, \eqref{eq:vh1}, and \eqref{eq:wscaling}, we obtain
$$
h^{-2}\|v\|_{0,h}^2 \lesssim (A_r v, v),
$$
which implies $\rho_{A_r}\gtrsim h^{-2}$.
\end{proof}

Now we are able to derive the following approximation property:
\begin{lemma}
Under the same assumptions as in Lemma \ref{lem:reducedspectralradius}, one has
$$
\|v_r - \Pi_r P_r v_r\|_{0,h}\lesssim \rho_{A_r}^{-1/2}\|v_r\|_A \qquad \textrm{for all }v_r\in V_r.
$$
\end{lemma}
\begin{proof}
By the triangular inequality and Equation \eqref{eq:Pvapproximability}, one has
$$
\begin{aligned}
\|v_r - \Pi_r P_r v_r\|_{0,h} 
  &= \|(I-P_0)v_b - \Pi_r P (I-P_0) v_b\|_{0,h} \\[2mm]
  &\lesssim \|(I-P_0)v_b - P(I-P_0)v_b\|_{0,h} + \|w - Q_h w\|_{0,h}\\
  &\lesssim h |(I-P_0)v_b|_{1,h} + h\|w\|_1,\\
  &\lesssim h |(I-P_0)v_b|_{1,h},
\end{aligned}
$$
where we conveniently denote $w = P(I-P_0)v_b \in \cV_h$ and use $\Pi_r w = Q_hw $. In the last step, we have used
$$
h \|w\|_1 \lesssim h|w|_1= h|P(I-P_0)v_b|_{1,h} \lesssim h|(I-P_0)v_b|_{1,h}.
$$

Combining the above and using Lemma \ref{lem:discretenorm-equivalence} give
$$
\|v_r - \Pi_r P_r v_r\|_{0,h} \lesssim h |(I-P_0)v_b|_{1,h} \lesssim h \|v_r\|_{A}.
$$
According to Lemma \ref{lem:reducedspectralradius}, $\rho_{A_r}=O(h^{-2})$. This completes the proof of the lemma.
\end{proof}

Finally, for variable efficient $\mathbb A$, if we denoted by $\bar{\mathbb A}$, the piecewise constant approximation of $\mathbb A$, then it is easy to show
$$
(\mathbb A \nabla_w v, \nabla_w v)
\lesssim
(\bar{\mathbb A}\nabla_w v, \nabla_w v)
\lesssim
(\mathbb A \nabla_w v, \nabla_w v), \quad \text{for all }v \in V_h.
$$
Therefore, a good preconditioner for the piecewise constant case will lead to a good preconditioner for the variable case.

We are able to claim that, the auxiliary space multigrid preconditioner for the reduced system \eqref{eq:reducedeq} again yields a preconditioner system with condition number of $O(1)$.
\begin{theorem}
Suppose we have a smoother $R$ and auxiliary solver $\mathcal B$ satisfying the property: for all $v\in V_r$, $w\in \cV_h$,
\begin{align*}
\rho_{A_r}^{-1} \ldp v,\, v\rdp \lesssim \ldp R v, \, v\rdp &\lesssim \rho_{A_r}^{-1} \ldp v,\, v \rdp,\\
    (\cA w,\, w) \lesssim (\cB \cA w,\, \cA w) &\lesssim (\cA w\, w).
\end{align*}
Let $B = R + \Pi \cB \Pi^t$ or defined implicitly by the relation $I-BA_r = (I-RA_r)(I-\Pi \cB \Pi^t)(I-RA_r)$. Then $B$ is symmetric and positive definite and
$\kappa (BA_r) \lesssim O(1)$.
\end{theorem}

\section{Numerical results} \label{sec:numerical}
In this section, we examine the effectiveness of the auxiliary space multigrid preconditioner using several numerical examples. In all numerical experiments, we use the symmetric Gauss-Seidel smoothers as $R$ and the multiplicative version of the multigrid preconditioner. It is known that the multiplicative version multigrid usually performs better than the corresponding additive version.
The simulation is implemented using the MATLAB software
package $i$FEM~\cite{Chen.L2008c}.

The matrix $A$ for the lowest order weak Galerkin discretization, i.e., $P_0-P_0$ element, is assembled and the matrix $\mathcal A$ for the auxiliary problem using $P1$ element is obtained through the triple product $\mathcal A = \Pi^tA\Pi$ where $\Pi: \mathcal V_h \to V_h$ is the simple average. By doing so, there is no need to repeat the assembling procedure to get $\mathcal A$ and the implementation is more algebraic. After that, the matrices in coarse levels are obtained by the triple product using the standard prolongation and restriction operators of linear elements on hierarchical meshes.  We use PCG with the auxiliary space multigrid preconditioner.
The stopping criteria for all iterations are reached when the relative error of the residual is less than $10^{-8}$. We report results for the original system and the reduced system, respectively.
Since the main purpose of these numerical results is to examine the efficiency of the auxiliary space preconditioner instead of
testing the accuracy of the weak Galerkin approximation, in the report we omit the approximation error part.
Because of this, there is no need to list the exact solution for each test problem.

\subsection*{Example 1} We first consider the Poisson equation defined on a circular mesh of the unit disk. The coarsest mesh is shown in Fig.~\ref{fig:2Dmesh} (a). A sequence of meshes are obtained by several uniformly regular refinements, i.e., a triangle is divided into four congruent four triangles by connecting middle points of edges, of the coarsest mesh. Results are summarized in Table \ref{example1}.

\subsection*{Example 2}
Next, we consider a variable coefficient problem with an oscillating coefficient:
  $$
  -\nabla\cdot \left( 2(2+\sin(10\pi x)\sin(10\pi y)) \nabla u\right) = f
  $$
  on $[0,1]\times[0,1]$. The coarsest mesh has size $h=1/4$ and is shown in Fig.~\ref{fig:2Dmesh} (b). Fourth order quadrature is used when assembling the stiffness matrix. Results are summarized in Table \ref{example2}.

\subsection*{Example 3}
We consider a test problem on an $L$-shaped domain obtained by
subtracting $[0,1]\times[-1,0]$ from $(-1,1)\times(-1,1)$. The Poisson equation on such a domain has $H^{3/2}$-regularity. Adaptive finite element method based on {\it a posteriori} error estimator constructed in~\cite{Chen2013} is used. A sample adaptive mesh obtained by bisection refinement is shown in Fig.~\ref{fig:2Dmesh} (c). For bisection grids, we apply coarsening algorithm developed in~\cite{Chen2010a} to obtain a hierarchy of meshes. In Table \ref{example3}, only results on some selected adaptive meshes are reported since the full list of adaptive meshes are long and the performance remains similar for all meshes.

\begin{figure}[htbp]
\label{fig:2Dmesh}
\subfigure[Initial grid of Example 1]{
\begin{minipage}[t]{0.33\linewidth}
\centering
\includegraphics*[width=3.1cm]{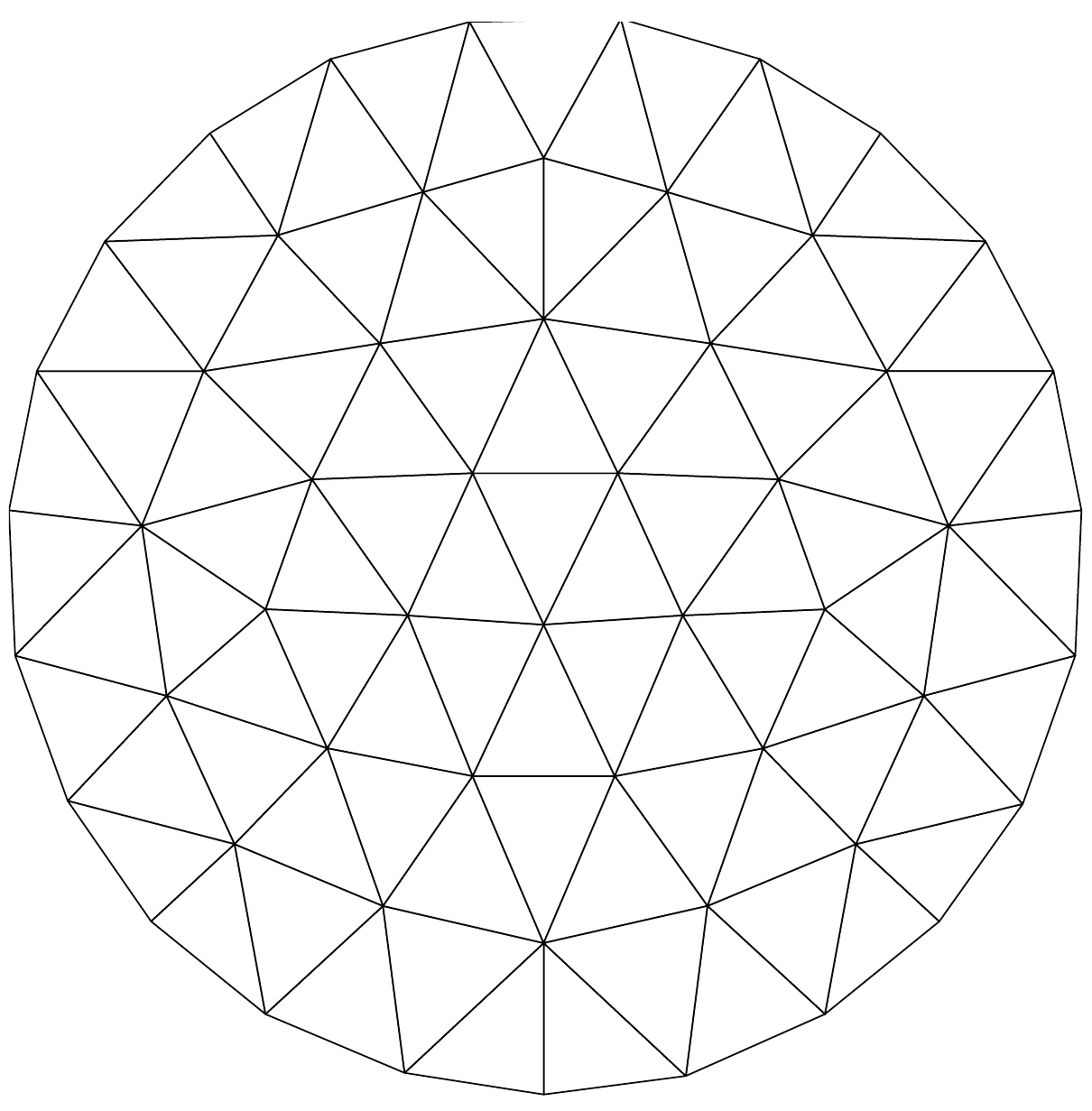}
\end{minipage}}
\subfigure[Initial grid of Example 2]
{\begin{minipage}[t]{0.32\linewidth}
\centering
\includegraphics*[width=3cm]{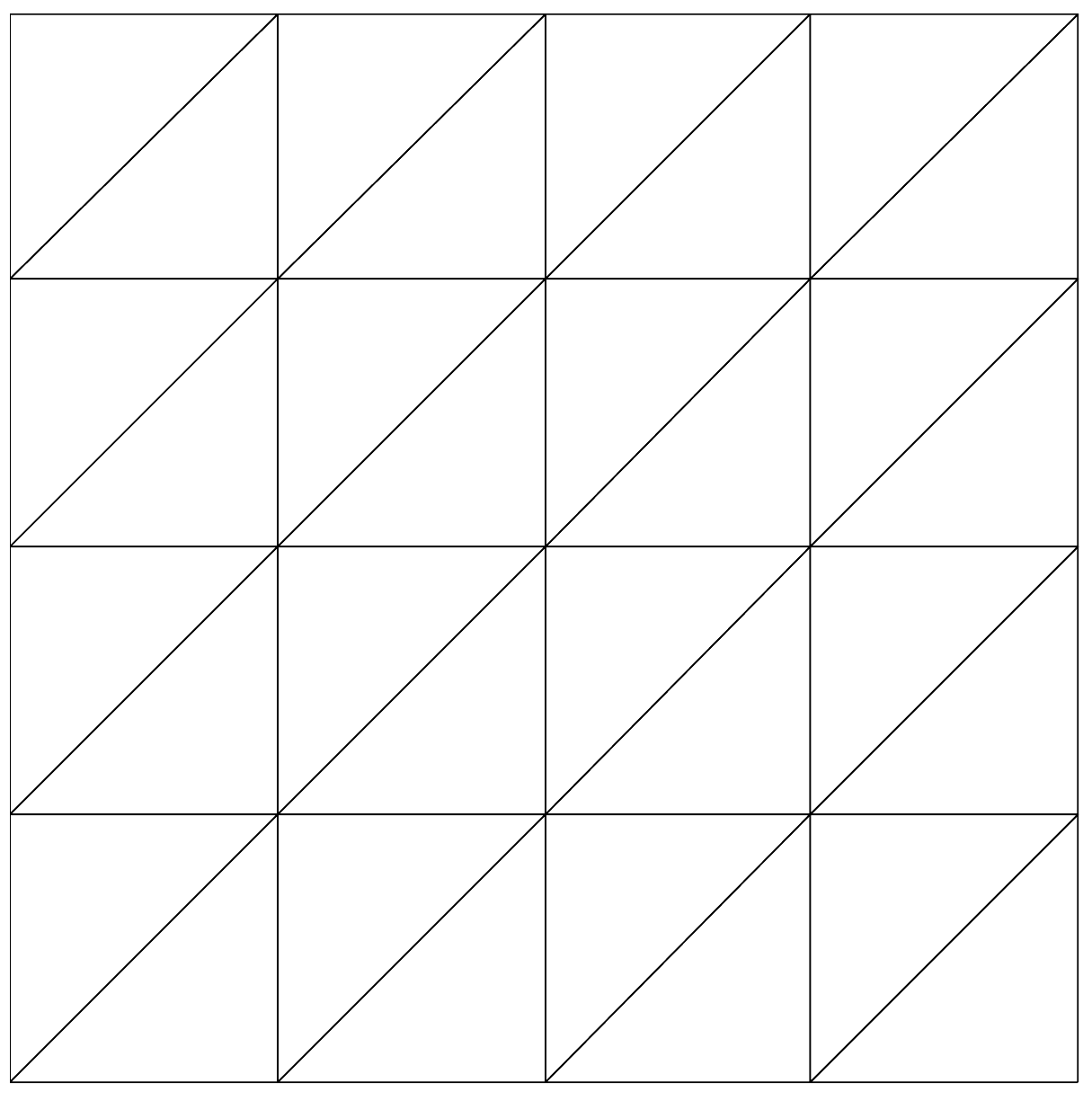}
\end{minipage}}
\subfigure[An adaptive grid of Example 3]
{\begin{minipage}[t]{0.34\linewidth}
\centering
\includegraphics*[width=3cm]{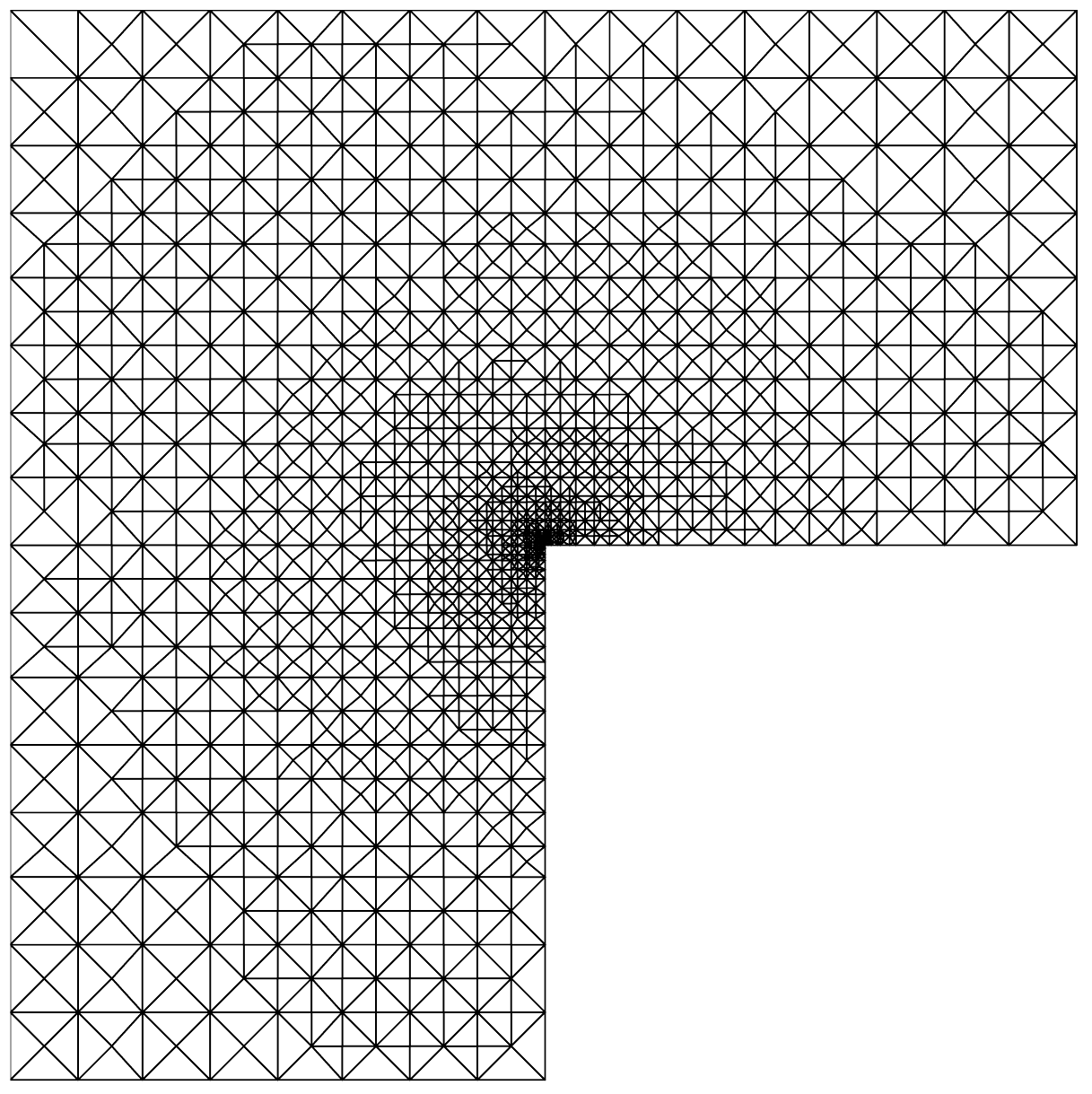}
\end{minipage}}
\caption{Meshes in Example 1, 2, 3.}
\end{figure}

\begin{table}[htdp]
\caption{Iteration steps and CPU time (in seconds) for Example 1. The left table is for the original system and the right table is for the reduced system.}
\begin{center}
\begin{tabular}{c c c }
\hline \hline
Dof & Steps & Time\\ \hline
  3446 &  13 &  0.052\\ \hline
 13692  & 13   & 0.11\\ \hline
 54584  & 13   & 0.41\\ \hline
217968  & 13  & 1.8\\ \hline
871136  & 13  &  8\\ \hline
\hline
\end{tabular}
\quad \quad
\begin{tabular}{c c c }
\hline \hline
Dof & Steps & Time\\ \hline
  2086 &  8 &  0.02 \\ \hline
  8252 &  8  & 0.053 \\ \hline
 32824  & 8   & 0.17 \\ \hline
130928 &  8  &  0.66 \\ \hline
522976 &  8   &  2.9 \\ \hline
\hline
\end{tabular}
\end{center}
\label{example1}
\end{table}

\begin{table}[htdp]
\caption{Iteration steps and CPU time (in seconds) for Example 2. The left table is for the original system and the right table is for the reduced system.}
\begin{center}
\begin{tabular}{c c c }
\hline \hline
Dof & Steps & Time\\ \hline
  1312 &  13 &  0.017 \\ \hline
  5184  & 13  & 0.048 \\ \hline
 20608  & 14  &  0.17 \\ \hline
 82176  & 14  &  0.63 \\ \hline
328192 &  14  &   2.7 \\ \hline
\hline
\end{tabular}
\quad \quad
\begin{tabular}{c c c }
\hline \hline
Dof & Steps & Time\\ \hline
   800  &  9 &  0.013 \\ \hline
  3136 &   9  & 0.036 \\ \hline
 12416  & 10 &  0.082 \\ \hline
 49408   & 9  &  0.27 \\ \hline
197120  &  9  &   1.1 \\ \hline
\hline
\end{tabular}
\end{center}
\label{example2}
\end{table}%

\begin{table}[htdp]
\caption{Iteration steps and CPU time (in seconds) for Example 3. The left table is for the original system and the right table is for the reduced system.}
\begin{center}
\begin{tabular}{c c c }
\hline \hline
Dof & Steps & Time\\ \hline
  256 &  13 &  0.0099 \\ \hline
  574  & 13  &  0.023 \\ \hline
 1091 &  13  &  0.035\\ \hline
 2177  & 13  &  0.058\\ \hline
 4398 &  13  &   0.11\\ \hline
 8642  & 13  &   0.16\\ \hline
10742 &  13 &     0.2\\ \hline
\hline
\end{tabular}
\quad \quad
\begin{tabular}{c c c }
\hline \hline
Dof & Steps & Time\\ \hline
 160  &  8 &  0.0088\\ \hline
  352  &  9 &   0.014\\ \hline
  663  & 10  &  0.023\\ \hline
 1317  & 10  &  0.036\\ \hline
 2656 &  10  &  0.067\\ \hline
 5206  &  9 &   0.091\\ \hline
 6470  &  8  &  0.094\\ \hline
\hline
\end{tabular}
\end{center}
\label{example3}
\end{table}%

\begin{figure}[htbp]
\label{fig:3Dmesh}
\subfigure[Initial grid of Example 4]{
\begin{minipage}[t]{0.45\linewidth}
\centering
\includegraphics*[width=4.2cm]{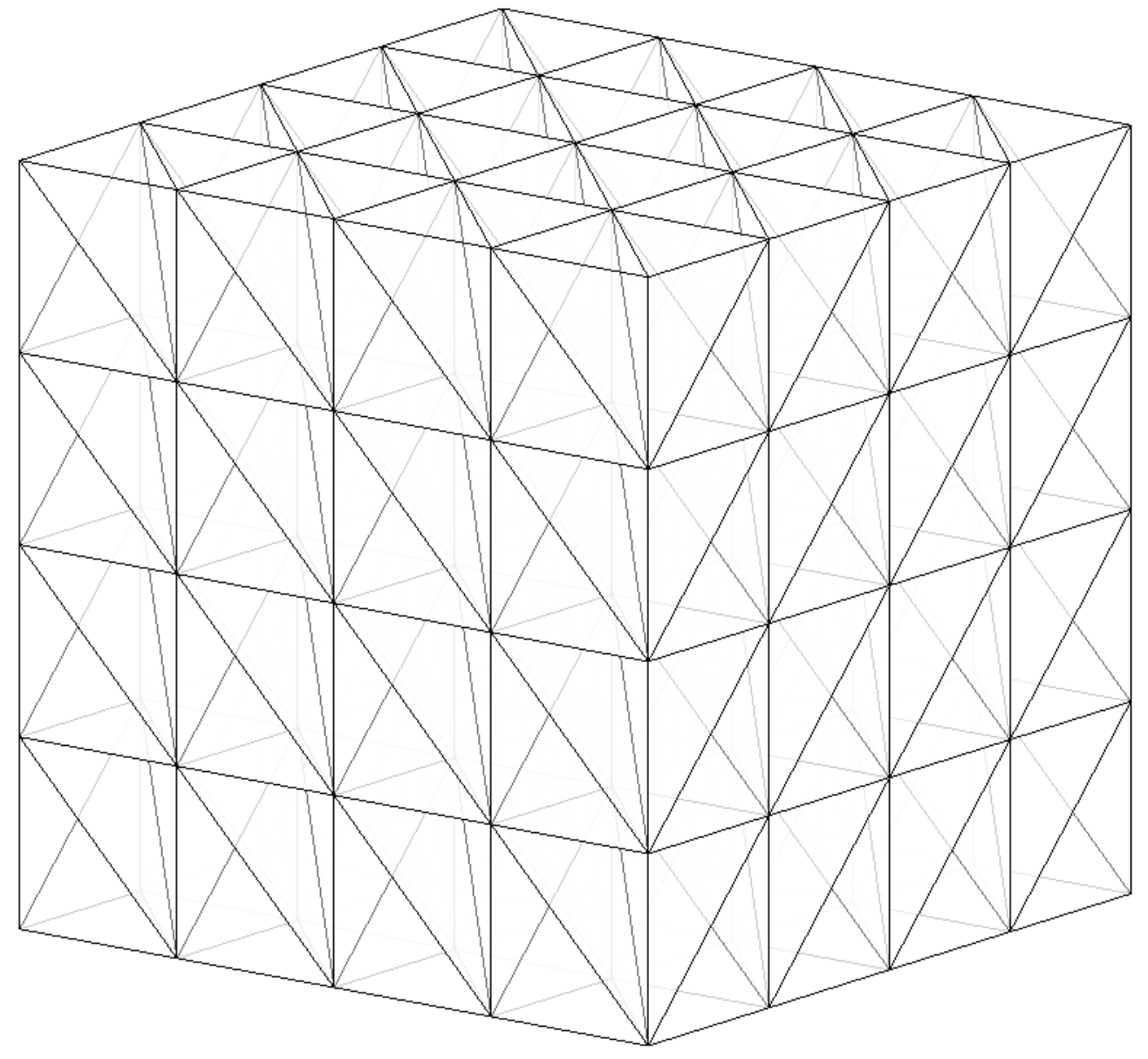}
\end{minipage}}
\subfigure[Domain of Example 5]
{\begin{minipage}[t]{0.45\linewidth}
\centering
\includegraphics*[width=4.1cm]{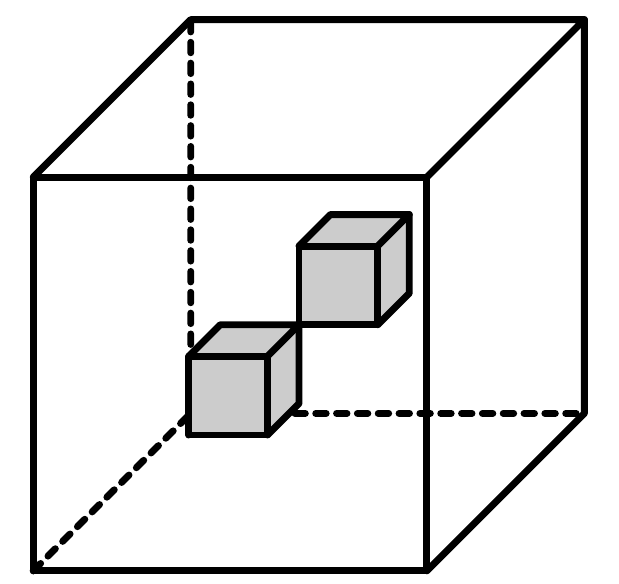}
\end{minipage}}
\caption{The coefficients $a_{1} =a_{2}=1$ in the gray domains $\Omega _1$ and $\Omega _2,$ and $a_{3} = \varepsilon$ in the rest of the domain.}
\end{figure}

\subsection*{Example 4} We consider the Poisson equation defined on the cube $\Omega = (-1,1)^3$. The coarsest mesh is shown in Fig.~\ref{fig:3Dmesh} (a). A sequence of meshes are obtained by several uniformly regular refinements, i.e., a tetrahedron is divided into $8$ small tetrahedron by connecting middle points of edges, of the coarsest mesh. Results are summarized in Table \ref{example4}.

\subsection*{Example 5} We consider the elliptic equation with jump coefficients~\cite{Xu.J1991,Xu.J;Zhu.Y2008}.
Let $\Omega = (-1,1)^3$ and the diffusion coefficient
$a(x)I$ be defined such that $a(x)$ is equal to the constants $a_1=a_2=1$ and $a_3=\varepsilon$
on the three regions $\Omega_1,\;\Omega_2$ and $\Omega_3$ respectively (see Figure~\ref{fig:3Dmesh} (b)), where
$$
  \Omega_1=(-0.5, 0)^3,\Omega_2=(0,0.5)^3
\; \hbox{ and }\; \Omega_3=\Omega\setminus (\overline{\Omega}_1\cup\overline{\Omega}_2).
$$
A sequence of meshes are obtained by several uniformly regular refinements of the coarsest mesh.

We choose $f=1$ and impose the following boundary conditions: Dirichlet conditions $$u_{\{-1\}\times[-1,1]\times[-1,1]}=0, \quad u_{\{1\}\times[-1,1]\times[-1,1]} =1,$$ and homogeneous Neumann boundary conditions on the remaining boundary.  We test the robustness of our solver as the coefficient $\epsilon$ changes. Only the reduced system is solved in this example.
Results are summarized in Table \ref{example5}.

\begin{table}[htdp]
\caption{Iteration steps and CPU time (in seconds) for Example 4. The left table is for the original system and the right table is for the reduced system.}
\begin{center}
\begin{tabular}{c c c }
\hline \hline
Dof & Steps & Time\\ \hline
   1248 &  16 &  0.018\\ \hline
   9600  & 18   &  0.1\\ \hline
  75264  & 18   &    1\\ \hline
 595968  & 19   &   10\\ \hline
4743168  & 19  & 100\\ \hline
\hline
\end{tabular}
\quad \quad
\begin{tabular}{c c c }
\hline \hline
Dof & Steps & Time\\ \hline
    864 &  11 &  0.058\\ \hline
   6528  & 12 &  0.054\\ \hline
  50688  & 13  &  0.41\\ \hline
 399360  & 13  &     4\\ \hline
3170304  & 13 &     36\\ \hline
\hline
\end{tabular}
\end{center}
\label{example4}
\end{table}%

\begin{table}
\caption{Iteration steps for Example 5. Only results for solving the reduced system is presented.}
\begin{center}
\begin{tabular}{c c c c c c}
\hline \hline
Dof & $\epsilon = 10^{-4}$ & $\epsilon = 10^{-2}$ & $\epsilon = 1$ & $\epsilon = 10^{2}$ & $\epsilon = 10^{4}$\\ \hline
864 & 36 & 22 & 13 & 13 & 13\\ \hline
6528 & 33 & 21 & 13 & 13 & 13\\ \hline
50688 & 32 & 21 & 13 & 13 & 13\\ \hline
399360 & 34 & 21 & 13 & 13 & 13\\ \hline
3170304 & 34 & 21 & 13 & 13 & 13\\ \hline
\hline
\end{tabular}
\end{center}
\label{example5}
\end{table}

From these experiments we may draw the following conclusions:
\begin{enumerate}
\item In all examples, the auxiliary space preconditioner works well for the linear system arising from discretization of the lowest order weak Galerkin method. The fluctuation of iteration steps of the PCG method applied to systems with different sizes is small which implies the condition number of the preconditioned system is uniformly bounded.

\item The solver for the reduced system is more efficient than the original system. The size of the reduced system is around two thirds of the original one and the time for solving the reduced system is around half of the original one. This shows the efficiency gained by working on the reduced system.

\item Although our theory is developed for quasi-uniform meshes, the third example indicates that our solver works well for adaptive grids and elliptic equations with less regularity.
\end{enumerate}






\begin{thebibliography}{99}
\bibitem{adams}
{\sc R. Adams and J. Fournier}, {\em Sobolev Spaces}, Academic
press, 2003.

\bibitem{Bramble93}
{\sc J.~Bramble}, {\em Multigrid methods},
Pitman Research Notes in Mathematics Series. Pitman, London, 1993.

\bibitem{Brenner}
{\sc S. Brenner and L. Ridgway},
{\em The Mathematical Theory of Finite Element Methods},
Springer, 2007.

\bibitem{Chen.L2008c}
{\sc L.~Chen.}
\newblock {iFEM: An Integrated Finite Element Methods Package in MATLAB}.
\newblock {\em Technical Report, University of California at Irvine}, 2009.

\bibitem{Chen2013}
{\sc L. Chen, J. Wang and X. Ye},
{\em A Posteriori Error Estimates for Weak Galerkin Finite Element Methods for Second Order Elliptic Problems},
J. Sci. Comput., DOI 10.1007/s10915-013-9771-3

\bibitem{Chen2010a}
{\sc L.~Chen and C.-S. Zhang.}
\newblock {A Coarsening Algorithm on Adaptive Grids by Newest Vertex Bisection
  and its Applications}.
\newblock {\em Journal of Computational Mathematics}, 28(6):767--789, 2010.


\bibitem{ciarlet}
{\sc P.G. Ciarlet}, {\em The Finite Element Method for Elliptic
Problems}, North-Holland, New York, 1978.

\bibitem{Grisvard85}
{\sc P. Grisvard},
{\em Elliptic problems in nonsmooth domains},
Pitman, Boston, 1985.

\bibitem{HuangWang}
{\sc W. Huang and Y. Wang}
{\em Discrete maximum principle for the weak Galerkin method for anisotropic diffusion problems},
arXiv:1401.6232, submitted.

\bibitem{LiXie}
{\sc Binjie Li, Xiaoping Xie}
{\em Multigrid weak Galerkin finite element method for diffusion problems}
arXiv:1405.7506

\bibitem{WG-biharmonic}
{\sc L. Mu, J. Wang, Y. Wang and X. Ye},
{\em A Weak Galerkin Mixed Finite Element Method for Biharmonic Equations},
arXiv:1210.3818v2, submitted.

\bibitem{mwy-wg-stabilization}
{\sc L. Mu, J. Wang, and X. Ye}, {\em Weak Galerkin finite element
methods on polytopal meshes}, arXiv:1204.3655v2, submitted to SINUM.

\bibitem{MuWangWangYe}
{\sc L. Mu, J.~Wang, Y.~Wang and X.~Ye}, {\em A computational study
of the weak Galerkin method for second order elliptic equations},
arXiv:1111.0618v1, Numerical Algorithms, accepted.

\bibitem{rt}
{\sc P. Raviart and J. Thomas}, {\em A mixed finite element method
for second order elliptic problems}, Mathematical Aspects of the
Finite Element Method, I. Galligani, E. Magenes, eds., Lectures
Notes in Math. 606, Springer-Verlag, New York, 1977.

\bibitem{WangYe_PrepSINUM_2011}
{\sc J.~Wang and X.~Ye}, {\em A weak Galerkin finite element method
for second-order elliptic problems}, arXiv:1104.2897v1,
Journal of Computational and Applied Mathematics, accepted.

\bibitem{wy-mixed}
{\sc J. Wang and X. Ye}, {\em A Weak Galerkin mixed finite element
method for second-order elliptic problems}, arXiv:1202.3655v1,
submitted to Math Comp.

\bibitem{Xu.J1991}
J.~Xu.
\newblock Counter examples concerning a weighted ${L}^{2}$ projection.
\newblock {\em Mathematics of Computation}, 57:563--568, 1991.

\bibitem{xu92}
{\sc J. Xu},
{\em Iterative Methods by Space Decomposition and Subspace Correction},
SIAM Review, 34 (1992), pp.~581--613.

\bibitem{xu}
{\sc J.Xu},
{\em The auxiliary space method and optimal multigrid preconditioning techniques for unstructured grids},
Computing, 56 (1996), pp.~215--235.

\bibitem{Xu.J;Zhu.Y2008}
J.~Xu and Y.~Zhu.
\newblock Uniform convergent multigrid methods for elliptic problems with
  strongly discontinuous coefficients.
\newblock {\em Mathematical Models and Methods in Applied Science}, 18(1):77
  --105, 2008.

\end{thebibliography}
\end{document}